\documentclass[a4paper,10pt]{amsart}
\usepackage[centertags]{amsmath}
\usepackage[english]{babel}
\usepackage{amsfonts}
\usepackage{amssymb}
\usepackage{amsthm}
\usepackage[usenames,dvipsnames]{color}
\usepackage{newlfont}
\usepackage{graphicx}
\usepackage{epstopdf}
\usepackage{cite}
\usepackage{bm}
\usepackage{blkarray}
\usepackage{multirow}

\newcommand{\xdownarrow}[1]{%
    {\left\downarrow\vbox to #1{}\right.\kern-\nulldelimiterspace}}

\newcommand{\Z}{\mathbb{Z}}

\newcommand{\Q}{\mathbb{Q}}
\newcommand{\R}{\mathbb{R}}

\newcommand{\F}{\mathbb{F}}

\newcommand{\bq }{\begin{equation}}
\newcommand{\eq }{\end{equation}}
\theoremstyle{plain}
\newtheorem{thm}{Theorem}[section]
\newtheorem{lem}[thm]{Lemma}
\newtheorem{lema}[thm]{Lemma}
\newtheorem{prop}[thm]{Proposition}
\newtheorem{proposition}[thm]{Proposition}

\newtheorem{cor}[thm]{Corollary}
\newtheorem{rem}[thm]{Remark}

\theoremstyle{definition}
\newtheorem{defn}[thm]{Definition}
\newtheorem{example}[thm]{Example}

\theoremstyle{example}

\title{on twists of smooth plane curves}

\author[E. Badr] {Eslam Badr}
\address{$\bullet$\,\,Eslam Essam Ebrahim Farag Badr}
\address{Departament Matem\`atiques, Edif. C, Universitat Aut\`onoma de Barcelona\\
08193 Bellaterra, Catalonia, Spain} \email{eslam@mat.uab.cat}
\address{Department of Mathematics,
Faculty of Science, Cairo University, Giza-Egypt}
\email{eslam@sci.cu.edu.eg}

\author[F. Bars] {Francesc Bars}
\address{$\bullet$\,\,Francesc Bars Cortina}
\address{Departament Matem\`atiques, Edif. C, Universitat Aut\`onoma de Barcelona\\
08193 Bellaterra, Catalonia, Spain} \email{francesc@mat.uab.cat}

\author[E. Lorenzo Garc\'ia]{Elisa Lorenzo Garc\'ia}
\address{$\bullet$\,\,Elisa Lorenzo Garc\'ia}
\address{Mathematisch Instituut, Universiteit Leiden, Niels Bohrweg 1,
	2333CA Leiden\\
	The Netherlands}
\address{Laboratoire IRMAR, Universit\'e de Rennes 1\\
	Campus de Beaulieu
	35042 Rennes cedex,
	France
} \email{elisa.lorenzogarcia@univ-rennes1.fr}

\thanks{E. Badr and F. Bars are supported by MTM2013-40680-P}



\begin{document}
%\eslam{},\francesc{},\elisa{}
\maketitle

\begin{small}
\begin{flushright}
    \textsc{To Pilar Bayer for her $70^{th}$ birthday}
  \end{flushright}
\end{small}

\begin{abstract} Given a smooth curve defined over a field $k$ that admits a non-singular plane
model over $\overline{k}$, a fixed separable closure of $k$, it
does not necessarily have a non-singular plane model defined over
the field $k$. We determine under which conditions this happens and
we show an example of such phenomenon: a curve defined over $k$ admitting plane models but none defined over $k$. Now, even assuming that such
a smooth plane model exists, we wonder about the existence of
non-singular plane models over $k$ for its twists. We characterize
twists possessing such models and we also show an example of a
twist not admitting any non-singular plane model over $k$. As a consequence,
we get explicit equations for a non-trivial Brauer-Severi surface.
Finally, we obtain a theoretical result to describe all the twists of
smooth plane curves with cyclic automorphism group having a
model defined over $k$ whose automorphism group is generated by a diagonal
matrix.
\end{abstract}

\begin{section}{Introduction}
Let $C$ be a smooth curve over a field $k$, i.e. a projective, non-singular and
geometrically irreducible curve defined over $k$. Let $\overline{k}$ be a fixed separable closure of $k$, the curve $C\times_k\overline{k}$ is denoted by $\overline{C}$, and its automorphism group by $\text{Aut}(\overline{C})$. We assume, once and for all, that $\overline{C}$ is non-hyperelliptic
of genus $g\geq3$. With the method exhibited in \cite{Lo} we can compute the twists of $C$; a twist of $C$ over $k$ is a smooth curve $C'$ over $k$ with a $\overline{k}$-isomorphism $\phi:\,\overline{C'}\rightarrow \overline{C}$. The set of twists of $C$ modulo $k$-isomorphisms, denoted by $\text{Twist}_k(C)$, is in one to one correspondence with the first Galois cohomology set $\text{H}^1(\text{Gal}(\overline{k}/k),\text{Aut}(\overline{C}))$
given by $[C']\mapsto \xi:\tau\mapsto\xi_{\tau}:=\phi\circ\, ^{\tau}\phi^{-1}$ where $\tau\in \text{Gal}(\overline{k}/k)$.
Given a cocycle $\xi\in\text{H}^1(\text{Gal}(\overline{k}/k),\text{Aut}(\overline{C}))$,
the idea behind the computation of equations for the twist, is finding a $\text{Gal}(\overline{k}/k)$-modulo isomorphism between
the subgroup generated by the image of $\xi$ in
$\text{Aut}(\overline{C})$ and a subgroup of a general linear group
$\text{GL}_n(\overline{k})$. After that, by making explicitly
Hilbert's Theorem $90$, we can compute an isomorphism $\phi:\,\overline{C'}\rightarrow \overline{C}$, and hence, we obtain equations for the twist. For non-hyperelliptic curves, see a description in \cite{Loth}, the canonical model gives a natural $\operatorname{Gal}(\overline{k}/k)$-inclusion
$\text{Aut}(\overline{C})\hookrightarrow\text{PGL}_g(\overline{k})$, but we can go further, the action gives a $\operatorname{Gal}(\overline{k}/k)$-inclusion $\text{Aut}(\overline{C})\hookrightarrow\text{GL}_g(\overline{k})$ which allows us to compute the twists.

Now consider a smooth $\overline{k}$-plane curve $C$ over $k$, i.e.
$C$ is a smooth curve over $k$ that admits a non-singular plane
model over $\overline{k}$. Therefore, $\overline{C}$ has a $g^2_d$
complete linear series which defines a map $\Upsilon:
\overline{C}\hookrightarrow \mathbb{P}_{\overline{k}}^2$, where
$\mathbb{P}^2_{\overline{k}}$ is the 2nd projective space over
$\overline{k}$. Moreover, $Image(\Upsilon)$ is defined by the zeroes
of a degree $d$ polynomial in $X,Y,Z$ with coefficients in
$\overline{k}$. Denote such a model by $F_{\overline{C}}(X,Y,Z)=0$,
in particular $g=\frac{1}{2}(d-1)(d-2)$. It is well-known that the
complete linear series $g^2_d$ is unique up to conjugation in
$\text{PGL}_3(\overline{k})$, the automorphism group of
$\mathbb{P}^2_{\overline{k}}$, see \cite[Lemma 11.28]{Book}.
Therefore, any $\overline{k}$-model of $C$ is defined by
$F_{P^{-1}\overline{C}}(X,Y,Z):=F(P(X,Y,Z))=0$ for some $P\in
\text{PGL}_3(\overline{k})$, observe that the $\overline{k}$-model
is an equation in $\mathbb{P}^2$ corresponding to the curve
$P^{-1}\overline{C}$ which is $\overline{k}$-isomorphic to
$\overline{C}$.
We say that $C$ is a smooth plane curve over $k$ if it is
$k$-isomorphic to $F_{Q^{-1}\overline{C}}(X,Y,Z)=0$ for some
$Q\in \text{PGL}_3(\overline{k})$ with $F_{Q^{-1}\overline{C}}\in
k[X,Y,Z]$.

The aim of this paper is making a study of the twists of smooth
$\overline{k}$-plane curves by considering the embedding
$\text{Aut}(\overline{C}) \hookrightarrow \text{PGL}_3(\overline{k})$
instead of the one given by the canonical model. If the curve $C$, or any of its twists over $k$, is a
smooth plane curve over $k$, we have an embedding of $Gal(\overline{k}/k)$-groups for its automorphisms group into $\text{PGL}_3(\overline{k})$.

This approach leads to two natural questions: the first one,
given a smooth $\overline{k}$-plane curve $C$ defined over a field
$k$, is it a smooth plane curve over $k$?; and secondly, if the
answer is yes, is every twist of $C$ over $k$ also a smooth plane curve over $k$?
For both questions the answer is no in general, it is not.
We obtain results for the curves for which the above questions always have an affirmative answer, and we show different examples concerning the negative general answer.
Interestingly, in the way to get these examples, we need to handle with non-trivial Brauer-Severi surfaces, and we are able to compute explicit equations of a non-trivial one. As far as we know, this is the first time that such equations are exhibited.

Moreover, for smooth plane curves defined over $k$ with a cyclic
automorphism group generated by a diagonal matrix, we provide a general theoretical result to
compute all its twists.
These families of smooth plane curves have already been studied by the
first two authors in \cite{BaBa2,BaBa3}. These families have genus arbitrarily high,
so the method in \cite{Lo} does not work for them.

\subsection{Outline}
The structure of this paper is as follows.
Section $2$ is devoted to the study of the minimal field $L$
where there exists a non-singular model over $L$ for a smooth $\overline{k}$-plane curve $C$ defined over $k$, i.e. that $C$ is $L$-isomorphic to $F_{Q^{-1}\overline{C}}(X,Y,Z)=0$ for some $Q\in \text{PGL}_3(\overline{k})$ with $F_{Q^{-1}\overline{C}}\in L[X,Y,Z]$. We prove that if the degree of a non-singular $\overline{k}$-plane model of $C$ is coprime with $3$, or $C$ has a $k$-rational point or the 3-torsion of the Brauer group of $k$ is trivial (in particular, if $k$ is a finite field), then the curve $C$ is a smooth plane over $k$ (i.e. admits a $k$-model): Theorem \ref{degree3} and Corollaries \ref{ratpoint}, \ref{ObsPlaneModel}. Moreover, we prove that a smooth plane model of $C$ always exists in a finite extension of $k$ of degree dividing $3$, see Theorem
\ref{ratplanemodel}. Section $2$ ends with an explicit example of a smooth $\overline{\Q}$-plane curve over $\mathbb{Q}$ which is not a smooth plane curve over $\mathbb{Q}$; however, we construct a smooth plane model over a degree $3$ extension of $\mathbb{Q}$.

In Section 3, we assume that $C$ is a smooth plane curve over $k$. We obtain Theorem \ref{lem1} characterizing the twists of $C$ which are also smooth plane curves over $k$. Moreover, we
construct a family of examples over $k=\Q$ for which a twist of $C$ does not admit a non-singular plane model over $\Q$. This
construction is not explicit because we do not provide equations of
such twists.

Section $4$ details an explicit example of a smooth
$\overline{\Q(\zeta_3)}$-plane curve over $\mathbb{Q}(\zeta_3)$
having a twist that does not possess such a model in the field
$\Q(\zeta_3)$, where $\zeta_3$ is a primitive 3rd root of unity.
Interestingly, we find the already mentioned explicit equations for
a non-trivial Brauer-Severi variety.

In Section $5$, we study the twists for smooth plane curve $C$ over
$k$, such that $\text{Aut}(\overline{C})$ is a cyclic group. We
prove that if $\text{Aut}(F_{P^{-1}\overline{C}})$ is represented in
$\text{PGL}_3(\overline{k})$ by a diagonal matrix, (where
$F_{P^{-1}\overline{C}}(X,Y,Z)$ is $k$-isomorphic to $C$) then all
the twists are diagonal, i.e. of the form
$F_{(PD)^{-1}\overline{C}}(X,Y,Z)=0$ with $D$ a diagonal matrix,
Theorem \ref{lem4}. We apply this result to some special families of curves,  see Corollary \ref{thm1}. We also construct an example of a curve $C$ that being
$\text{Aut}(F_{P^{-1}\overline{C}})$ cyclic (but not diagonal) has all the twists not diagonal.

\subsection{Notation and conventions}
We set the following notations, to be used
throughout. By $k$ we denote a field, $\overline{k}$ is a separable closure of
$k$ and $L$ is an extension of $k$ inside $\overline{k}$. By
$\zeta_n$ we always mean a fixed primitive $n$-th root of unity
inside $\overline{k}$ when the characteristic of $k$ is coprime with
$n$. We write $\text{Gal}(L/k)$ for the Galois group of $L/k$, and we consider left actions. The  Galois cohomology sets of a
$\text{Gal}(L/k)$-group $G$ are denoted by
$\text{H}^i(\text{Gal}(L/k), G)$ with $i\in\{0,1\}$ respectively.
For the particular case $L=\overline{k}$, we use $\operatorname{G}_k$ instead of
$\text{Gal}(\overline{k}/k)$ and $\text{H}^1(k,G)$ instead of
$\text{H}^1(\operatorname{Gal}(\overline{k}/k),G)$. Furthermore, $\operatorname{Br}(k)$ denotes
the Brauer group of $k$ whose elements are the Brauer equivalence
classes of central simple algebras over $k$. Let $\operatorname{Az}_n^k$ denote the
set of all equivalence classes of central simple algebras of
dimension $n^2$ over $k$ modulo $k$-algebras isomorphisms (each of them splits in a separable extension of degree
$n$ of $k$). The $n$-torsion of $\operatorname{Br}(k)$ coincides with
$\operatorname{Az}_n^k$ and there is a bijection between $\operatorname{Az}_n^k$ with
$\operatorname{H}^1(k,\text{PGL}_n(\overline{k}))$ (see \cite[Corollary3.8]{Ja}).

We use the SmallGroup Library-GAP \cite{GAP} notation where $\operatorname{GAP}(N,r)$ represents the group of order $N$ appearing in the $r$-th position in such atlas. For cyclic groups, we use the standard notation $\Z/n\Z$.

By a smooth curve over $k$ we mean a projective, non-singular and
geometrically irreducible curve defined over $k$, and it will be
denoted by $C$ or $C_k$. As usual $\overline{C},\,\text{Aut}(\overline{C})$ and $g(\overline{C})$ denote
$C\times_k\overline{k}$,
the automorphism group of $\overline{C}$, and
its genus. We assume, once and for all, that
$g(\overline{C})\geq 2$.

By a smooth $\overline{k}$-plane curve $C$ over $k$ we mean a smooth
curve over $k$ admitting a non-singular plane model
$F_{\overline{C}}(X,Y,Z)=0$ over $\overline{k}$ of degree $d\geq 4$. We say that $C$ is a plane curve of degree $d$.
Note that any other plane model has the form
$F_{P^{-1}\overline{C}}(X,Y,Z)=0$ for some $P\in
\text{PGL}_3(\overline{k})$, where
$F_{P^{-1}\overline{C}}(X,Y,Z):=F_{\overline{C}}(P(X,Y,Z))$. Moreover, the automorphism group
$\text{Aut}(F_{P^{-1}\overline{C}})$ of
$F_{P^{-1}\overline{C}}(X,Y,Z)=0$ is a finite subgroup of
$\text{PGL}_3(\overline{k})$, and it is equal to $P^{-1}
\text{Aut}(F_{\overline{C}}) P$. Observe that the natural map of
smooth plane curves over $\overline{k}$:  $\overline{C}\xrightarrow{P^{-1}} P^{-1}\overline{C}\xrightarrow{Q^{-1}}
Q^{-1}P^{-1}\overline{C}=(PQ)^{-1}\overline{C}$ corresponds to
$\{F_{\overline{C}}=0\}\xrightarrow{P^{-1}}
\{F_{P^{-1}\overline{C}}=0\}\xrightarrow{Q^{-1}}\{F_{Q^{-1}P^{-1}\overline{C}}=0\}=\{F_{(PQ)^{-1}\overline{C}}=0\}$,
where $P,Q\in \text{PGL}_3(\overline{k})$.

We denote by $\mathbb{P}^r_L$ the $r$-th
projective space over the field $L$. A linear transformation $A=(a_{i,j})$ of $\mathbb{P}^2_{L}$ is often written as $[a_{1,1}X+a_{1,2}Y+a_{1,3}Z:a_{2,1}X+a_{2,2}Y+a_{2,3}Z:
a_{3,1}X+a_{3,2}Y+a_{3,3}Z]$.

Given a smooth $\overline{k}$-plane curve $C/k$, we say that $C$
admits a non-singular plane model over $L$ if there exists $P\in
\text{PGL}_3(\overline{k})$ with $F_{P^{-1}\overline{C}}(X,Y,Z)\in
L[X,Y,Z]$, and such that $C$ and $F_{P^{-1}\overline{C}}(X,Y,Z)=0$
are isomorphic over $L$. If a smooth $\overline{k}$-plane curve $C$
over $k$ admits a non-singular plane model
$F_{P^{-1}\overline{C}}=0$ over $k$ which is isomorphic to $C$, we
say that $C$ is a smooth plane curve over $k$ and, then we identify,
by an abuse of notation, $C$ with the plane model
$F_{P^{-1}\overline{C}}=0$ and $\text{Aut}(\overline{C})$ with
$\text{Aut}(F_{P^{-1}\overline{C}})$ as a fixed finite subgroup of
$\text{PGL}_3(\overline{k})$.

\subsection*{Acknowledgments}
The authors would like to thank Xavier Xarles for his suggestions
and comments concerning our work and for clarifying different
aspects of his work in \cite{RoXa}. It is also our pleasure to thank
Ren\'e Pannekoek for e-mailing us his Master thesis \cite{Pa} and
for his useful comments on Brauer-Severi surfaces. Finally, we thank
Christophe Ritzenthaler who pointed us different interesting
comments.

\begin{section}{The field of definition of a non-singular plane model}

In this section, we prove that if $C$ is a smooth
$\overline{k}$-plane curve defined over $k$, (i.e. $C$ is a smooth
projective curve over $k$ that admits a non-singular plane model
over $\overline{k}$), then it is always possible to find a
non-singular plane model defined over an extension $L/k$ of degree
dividing $3$. Moreover, if a (or any) smooth plane model of $C$ over
$\overline{k}$ has degree coprime with $3$, we prove that we can
always find a non-singular plane model defined over the base field
$k$, i.e that $C$ is a smooth plane curve over $k$. We also provide an
example of a smooth curve defined over $\mathbb{Q}$ that does not
admit a smooth plane model over $\mathbb{Q}$, but that it does over
a Galois extension of $\mathbb{Q}$ of degree $3$.

We first recall that, a Brauer-Severi variety $D$ over $k$ of
dimension $r$ is a smooth projective variety such that the variety
$D\otimes_k\overline{k}$ over $\overline{k}$ is isomorphic to the
projective space $\mathbb{P}^r_{\overline{k}}$ of dimension $r$ over
$\overline{k}$, and is well-known \cite[Corollary 4.7]{Ja} that the
Brauer-Severi varieties over $k$ of dimension $r$, up to
$k$-isomorphism, are in bijection with
$\text{H}^1(k,\text{PGL}_{r+1}(\overline{k}))=\text{H}^1(k,\text{Aut}_{\overline{k}}(\mathbb{P}^r_{\overline{k}}))$.

 Ro\'e and Xarles proved the following result
in \cite[Corollary 6]{RoXa}.

\begin{thm}[Ro\'e-Xarles]\label{proprx} Let $C$ be a smooth $\overline{k}$-plane curve defined over $k$.
Let $\Upsilon:\overline{C}\hookrightarrow
\mathbb{P}^2_{\overline{k}}$ be a morphism given by (the unique) $g^2_d$-linear system over $\overline{k}$, then there exists
a Brauer-Severi variety $D$ (of dimension two) defined over $k$,
together with a $k$-morphism $g:C\hookrightarrow D$ such that
$g\otimes_k\overline{k}:\overline{C}\rightarrow\mathbb{P}^2_{\overline{k}}$
is equal to $\Upsilon$.
\end{thm}

 The idea of the proof of Theorem \ref{proprx}, that will be used in Section \ref{Sec21} is as follows: a $\overline{k}$-plane model of the curve $C$ defines a
1-cocycle in $\operatorname{H}^1(k,\operatorname{PGL}_3(\overline{k}))$ by the $g^2_d$-linear series, and the corresponding twist $\iota:
\mathbb{P}^2_k\rightarrow D$, maps the $\overline{k}$-plane model of
the curve $C$ into a smooth curve over $k$ in the Brauer-Severi
variety $D$ over $k$.

From the above result, one obtain remarkable consequences.

\begin{cor}\label{ratpoint} Let $C$ be a smooth $\overline{k}$-curve over $k$. Assume that $C$ has a
$k$-rational point, i.e. $C(k)$ is not-empty. Then $C$ admits a non-singular
plane model over $k$.
\end{cor}
\begin{proof} It is well-known \cite[Prop.4.8]{Ja}, that a
Brauer-Severi variety over $k$ of dimension $n$ with a $k$-rational point is
isomorphic over $k$ to $\mathbb{P}^n_k$. By Theorem \ref{proprx}, the
map $g:C_k\rightarrow D\cong\mathbb{P}^2_k$ defined over $k$ defines
a non-singular plane model of $C$ over $k$.
\end{proof}

\begin{cor}\label{ObsPlaneModel} Consider a field $k$ such that $\operatorname{Br}(k)[3]$ is trivial, where $\operatorname{Br}(k)[3]$ denotes
the 3-torsion of $\operatorname{Br}(k)$. Then any smooth plane curve $C$ over $k$,
 admits a non-singular plane model over $k$, and in particular any twist
of $C$ over $k$ admits also a non-singular plane model over $k$.
\end{cor}
\begin{proof} A non-trivial Brauer-Severi surface over $k$
corresponds to a non-trivial 3-torsion element of $\operatorname{Br}(k)$ by the
well-known result \cite[Corollary 3.8]{Ja} concerning a bijection
between $\operatorname{H}^1(k,\text{PGL}_n(\overline{k}))$ with $\operatorname{Az}_n^k$.
Therefore, if this group is trivial, by Theorem \ref{proprx}, the
$g_2^d$-system factors through $g:C_k\hookrightarrow \mathbb{P}^2_k$
and all of them are defined over $k$. Hence, they define a plane model
of $C$ over $k$.
\end{proof}

\begin{rem} For a finite field $k$ or $k=\R$, it is well-known that $\operatorname{Br}(k)[3]$ is
trivial. Therefore any smooth $\overline{k}$-plane curve over such
fields $k$ admits always a non-singular plane model over $k$.
\end{rem}

\begin{thm}\label{ratplanemodel} Let $C$ be a smooth plane curve defined over $k$,
then it admits a non-singular plane model over an $L$ such that
$[L:k]\mid 3$, i.e. $\exists P\in \operatorname{PGL}_3(\overline{k})$ for which $F_{P^{-1}\overline{C}}\in L[X,Y,Z]$ and such that $C$ and
$F_{P^{-1}\overline{C}}=0$ are $L$-isomorphic.
\end{thm}
\begin{proof} From Theorem \ref{proprx}, we have a $k$-morphism of
$C$ to a Brauer-Severi surface $D$ over $k$. By \cite[Corollary
3.8]{Ja}, $D$ corresponds to an element in $\operatorname{Az}_3^k$, i.e. a central
simple algebra over $k$ of dimension $9$ which is split (if it is
not trivial) by a degree $3$ Galois extension $L/k$. Therefore,
$D\otimes_k L$ corresponds to the trivial element in
$\text{H}^1(\text{Gal}(\overline{k}/L),\text{PGL}_3(\overline{k}))$,
by theorem \ref{parametrization}. Thus, $D\otimes_kL\cong
\mathbb{P}^2_L$ over $L$. Then
$$g\otimes_kL:C\otimes_kL\hookrightarrow \mathbb{P}^2_L$$
are all defined over $L$, and we have a non-singular plane model of
$C$ over $L$. Lastly, because all non-singular plane models of $C$
over $\overline{k}$ are of the form
$F_{P^{-1}\overline{C}}(X,Y,Z)=0$ for some
$P\in\text{PGL}_3(\overline{k})$, we then deduce the last statement.
\end{proof}

The following result is a particular case of an argument
by Ro\'e and Xarles in \cite{RoXa} following
Ch\^{a}telet \cite{Cha}.

\begin{thm}\label{degree3} Let $C$ be a smooth $\overline{k}$-plane curve defined over
$k$ of degree $d$ coprime
with 3. Then $C$ is a smooth plane curve over $k$.
\end{thm}
\begin{proof}
By the results of the previous section, Brauer-Severi surfaces over
$k$ corresponds to elements of
$\text{H}^1(k,\text{PGL}_3(\overline{k}))$, hence to $\operatorname{Az}_3^k$. 
Thus, they are elements of the Brauer group $\operatorname{Br}(k)$ of the field $k$ of order
dividing 3.

Moreover, if $D$ is a Brauer-Severi surface over a field $k$, then
its class $[D]$ in the Brauer group $\operatorname{Br}(k)$ verifies that in the
exact sequence
$$\operatorname{Pic}(D)\rightarrow \operatorname{Pic}(D\otimes_k\overline{k})\cong\Z \rightarrow
\operatorname{Br}(k),$$ the last map sends $1$ to $[D]$, and hence the image of
some generator of $\operatorname{Pic}(D)$ is equal to $m$, where $m$ is the order
of $[D]$. Consequently, $m$ divides $3$, as the order of $[D]$ does.
Now, if $C$ is a curve over $k$ in $\operatorname{Pic}(D)$ such that $\overline{C}$
has a non-singular plane {model} of degree $d$, then the image of
${C}$ in $\operatorname{Pic}(D\otimes_k\overline{k})\cong\Z$ is equal to the degree
$d$. Therefore, if $d$ is coprime with $3$, we get $m=1$, and $D$ is
the projective plane $\mathbb{P}^2_k$ (see \cite[Theorem 13]{RoXa}
for a more general statement on hypersurfaces in Brauer-Severi
varieties).

\end{proof}

\begin{cor} \label{deg3tw} Let $C$ be a smooth $\overline{k}$-plane curve defined over $k$ of degree $d$ coprime with $3$.
Then, every twist $C'\in \operatorname{Twist}_k(C)$ is a smooth plane curve
over $k$.
\end{cor}
\begin{proof} It follows, by our assumption, that every twist of $C$ over $k$ admits a non-singular plane model over $\overline{k}$ of degree $d$,
coprime with $3$. Hence, non-singular plane models over $k$ exist for twists of $C$ over $k$, by Theorem \ref{degree3}.

\end{proof}

\subsection{An example of a smooth $\overline{\Q}$-plane curve over $\Q$ which is not a smooth plane curve over $\Q$}\label{Sec21}

Following the proof of Theorem \ref{proprx}, in order to construct $C$ a smooth
$\overline{k}$-plane curve over $k$ which is not a smooth plane
curve over $k$, we need to construct a 1-cocycle in
$\operatorname{H}^1(k,\operatorname{PGL}_3(\overline{k}))$ corresponding to $C$, which is
not-trivial. By a result of Wedderbum \cite{We} all the elements of
$\operatorname{H}^1(k,\operatorname{PGL}_3(\overline{k}))$ are cyclic algebras. For the sake of
completeness we recall the following definition and results.

\begin{defn}\label{defn1} Let $L/k$ be a cyclic extension of degree $n$ with $\text{Gal}(L/k)=\langle\sigma\rangle$, and fix an isomorphism
$\chi:\text{Gal}(L/k)\rightarrow\Z/n\Z$ such that $\chi(\sigma)=\overline{1}$. Given $a\in
k^*$, we consider a k-algebra $(\chi,a)$ as follows: As an additive group, $(\chi,a)$ is an $n$-dimensional vector space over $L$ with basis $1,e,\ldots, e^{n-1}$:
$$(\chi,a):=\bigoplus_{1\leq i\leq n-1} L e^i.$$
Multiplication is given by the relations:  $e\,.\,\lambda=\sigma(\lambda)\,.\,e$ for $\lambda\in L$, and $e^n=a$. Such $(\chi,a)$ becomes a central simple
algebra of dimension $n^2$ over $k$ which splits in $L$ (see
\cite[\S2]{Ten}), and it is called the cyclic algebra associated to the character $\chi$ and the element $a\in k$.
\end{defn}

\begin{thm}\label{parametrization}(\cite[Corollary 3.8]{Ja}) Let $k$ be a field, then there is a bijection between the sets $\operatorname{Az}_n^k$ and $\operatorname{H}^1(k,\operatorname{PGL}_n(\overline{k}))$.
The elements of $\text{Az}_{3}^k$ are given by cyclic
algebras of the form $(\chi, a)$ as in Definition \ref{defn1} with
$n=3$. The assignment
$$(\chi,a)\in \operatorname{Az}_3^k\mapsto
\operatorname{inf}(\{A_{\tau}\}_{\tau\in \operatorname{Gal}(L/k)})\in \operatorname{H}^1(k,
\operatorname{PGL}_3(\overline{k}))$$ is given by $f(\lambda\otimes
1)=diag(\lambda,\sigma(\lambda),\sigma^2(\lambda))$ and $f(e\otimes
1)=\left(\begin{array}{ccc}
0&0&a\\
1&0&0\\0&1&0\end{array}\right)$. Here $\operatorname{inf}$
denotes the inflation map in Galois cohomology.

Moreover, $(\chi,a)\in\operatorname{Az}_3^k$ is the trivial $k$-algebra, if
and only if $a$ is the norm of an element of $L$, where $L$ is the
associated cyclic extension of $k$ of degree 3 of $(\chi,a)$.
\end{thm}

We now construct the example.

Let us consider $\mathbb{Q}_f$ the splitting field of the polynomial $f(t)=t^3+12t^2-64$.
It is an irreducible polynomial and the discriminant of $f$ is $(2^63^2)^2$, then $\text{Gal}(\mathbb{Q}_f/\mathbb{Q})\simeq \Z/3\Z$, moreover, as we can check with Sage, the discriminant of the field $\mathbb{Q}_f$ is a power of 3, and the prime 2 becomes inert in $\Q_f$.

Let us denote the roots of $f$ by $a,b,c$ in a fixed algebraic closure of $\Q$, and let us call $\sigma$ the element in the Galois group that acts by sending $a\rightarrow b\rightarrow c$.

\begin{prop} The smooth plane curve over $\mathbb{Q}_f$
    $$
    C:\,64Z^6+abY^6+aX^6+8Y^3Z^3+\frac{ab}{8}X^3Y^3+aZ^3X^3=0,
    $$ has $\Q$ as a field of definition, but it does not admit a plane
    non-singular model over $\Q$.
\end{prop}
\begin{proof}
    The matrix
$$  \phi=\left(\begin{array}{ccc}0&0&2\\1&0&0\\0&1&0\end{array}\right)$$
    defines an isomorphism $\phi:\,^{\sigma}C\rightarrow C$. This isomorphism $\phi$ satisfies the Weil cocycle condition
    \cite{We}
    ($\phi_{\sigma^3}=\phi_{\sigma}^3=1$), we therefore obtain that the curve is defined over $\mathbb{Q}$, and that there exists
    an isomorphism $\varphi_0:C_{\Q}\rightarrow C$
    where $C_{\Q}$ is a rational model such that $\phi=\varphi_0\,^{\sigma}\varphi_{0}^{-1}\in \text{PGL}_3(\Q)$.
The assignation $\phi_{\tau}:=\varphi_0\,^{\tau}\varphi_{0}^{-1}$
defines an element of
$\text{H}^1(\text{Gal}(\Q_f/\Q),\text{PGL}_3(\Q_f))$, by Theorem
\ref{parametrization}, this cohomology element is non-trivial
because $2$ is not a norm of an element of $\Q_f$ (since $2$ is
inert in $\Q_f$). Therefore $\varphi_0$ is not given by an element
of $\text{PGL}_3(\Q_f)$, or of $\text{PGL}_3(\overline{\Q})$ because
the cohomology class by the inflation map is not trivial.
Therefore the curve $C$ over $\Q$ does not admits a non-singular
plane model over $\Q$ (because if there is a non-singular plane
model over $\Q$, such model would be of the form $F_{(
PQ)^{-1}{\overline{C}}}(X,Y,Z)=0$ for some $P\in
\text{PGL}_3(\overline{\Q})$ where $F_{Q^{-1}\overline{C}}(X,Y,Z)=0$
a non-singular model over $\Q_f$, therefore $\varphi_0$ would be
given by $P\in \text{PGL}_3(\overline{\Q})$ which is not).
\end{proof}

\begin{rem} We have just seen an example of a curve defined over a field $k$ not admitting a particular model (a plane one) over the same field.
For hyperelliptic models, we find such examples after Proposition
$4.14$ in \cite{LeRi}. In \cite[chp. 5,7]{Hug}, there are also
examples of hyperelliptic curves and smooth plane curves where the
field of moduli is not a field of definition, so, in particular,
there are not such models defined over the fields of moduli.
\end{rem}

\end{section}

\begin{section}{On twists of plane models defined over $k$}\label{Sec4}

In this section, we assume, once and for all, that $C$ is a smooth
plane curve defined over $k$, that is, that $C$ is given
by an equation $F_{\overline{C}}=0$ with $F_{\overline{C}}\in
k[X,Y,Z]$. We characterize when all the twists of
$C$ are a smooth plane curve over $k$, and we give a (non-explicit)
example of a family of such curves $C$ having twists which are not a
smooth plane curve over $k$, i.e. not admitting a smooth plane model
over $k$.

\begin{thm}\label{lem1} Let $C$ be a smooth plane curve over $k$ which we identify with the plane
non-singular model $F_{\overline{C}}(X,Y,Z)=0$ with
$F_{\overline{C}}[X,Y,Z]\in k[X,Y,Z]$. Then there exists a natural
map
$$\Sigma: \operatorname{H}^1(k,\operatorname{Aut}(F_{\overline{C}}))\rightarrow
\operatorname{H}^1(k,\operatorname{PGL}_3(\overline{k})),$$ defined by the inclusion
$\operatorname{Aut}(F_{\overline{C}})\subseteq\operatorname{PGL}_3(\overline{k})$ as
$\operatorname{G}_k$-groups. The kernel of $\Sigma$ is the set of all twists of $C$
that are smooth plane curves over $k$. Moreover, any such twist is
obtained through an automorphism of $\mathbb{P}^{2}_{\overline{k}}$, that is, the
twist is $k$-isomorphic to
$F_{M^{-1}\overline{C}}(X,Y,Z):=F_{\overline{C}}(M(X,Y,Z))\in
k[X,Y,Z]$ for some $M\in \operatorname{PGL}_3(\overline{k})$.
\end{thm}

\begin{proof} The map is clearly well-defined. If a twist $C'$ admits a non-singular plane model
{$F_{\overline{C'}}$} over $k$, the isomorphism from
$F_{\overline{C'}}$ to $F_{\overline{C}}$ is then given by an
element $M\in \text{PGL}_3(\overline{k})$ (as any isomorphism
between two non-singular plane curves of degrees $>3$ is given by a
linear transformation in $\mathbb{P}^2_{\overline{k}}$, \cite{Chang}). Hence, the
corresponding $1$-cocycle $\sigma\mapsto M\,^{\sigma}M^{-1}\in
\operatorname{Aut}(F_{\overline{C}})$ becomes trivial in
$\text{H}^1(k,\text{PGL}_3(\overline{k}))$. Conversely, if a twist $C'$ is
mapped by $\Sigma$ to the trivial element in
$\text{H}^1(k,\text{PGL}_3(\overline{k}))$, then this twist is given by
a $\overline{k}$-isomorphism $\varphi:
F_{\overline{C}}\rightarrow C'$ defined by a matrix
$M\in\text{PGL}_3(\overline{k})$ that trivializes the
cocycle and such an $M$ produces a
non-singular plane model defined over $k$.
\end{proof}

\begin{rem} We can reinterpret the map $\Sigma$ in Theorem \ref{lem1} as the map that sends a twist $C'$ to the Brauer-Severi variety $D$ in Theorem
\ref{proprx}. But in order to define a natural map
$\Sigma':\operatorname{Twist}_k(C)\rightarrow \operatorname{H}^1(k,\operatorname{PGL}_3(\overline{k}))$ for
$C$ a smooth $\overline{k}$-plane curve over $k$, we need that
$\operatorname{Aut}(\overline{C})$ has a natural inclusion in
$\operatorname{PGL}_3(\overline{k})$ as $\operatorname{G}_k$-modules. This is also possible
if exists $P\in \operatorname{PGL}_3(\overline{k})$ where
$F_{P^{-1}\overline{C}}(X,Y,Z)\in k[X,Y,Z]$ because the inclusion
$\operatorname{Aut}(F_{P^{-1}\overline{C}})\subseteq \operatorname{PGL}_3(\overline{k})$ is of
$\operatorname{G}_k$-modules and defines a map
$\operatorname{Twist}_k(C)=\operatorname{H}^1(k,\operatorname{Aut}(F_{P^{-1}\overline{C}}))\rightarrow
\operatorname{H}^1(k,\operatorname{PGL}_3(\overline{k}))$.
\end{rem}

\begin{rem} Consider a smooth plane curve $C$ defined over $k$ of degree $d$ coprime with $3$ or such that $\operatorname{Br}(k)[3]$ is trivial. Then $\Sigma$ in Theorem
\ref{lem1} is the trivial map by Corollaries
\ref{deg3tw} and \ref{ObsPlaneModel}.
\end{rem}

\begin{rem}\label{rem4.3}
Theorem \ref{lem1} can be used to improve the algorithm for
computing twists for non-hyperelliptic curves, see \cite{Lo} or
\cite[Chp.1]{Loth}, for the special case of non-singular plane
curves. If $\Sigma$ is trivial in Theorem \ref{lem1}, then we can work with matrices
in $\operatorname{GL}_3(\overline{k})$ instead of in $\operatorname{GL}_g(\overline{k})$. 

In the in progress Ph.D thesis of the first author \cite{Es}, we use this
improvement to compute the twists of some particular families of
smooth plane curves over $k$.
\end{rem}

\subsection{Twists of smooth plane curve over $k$ which are not
smooth plane curves over $k$}

We construct a family of smooth plane curves over $\Q$ but some of
its twists are not smooth plane curves over $\Q$. This construction
is not explicit in the sense that we do not construct the equations
of the twist and the Brauer-Severi surface where the twist lives,
see next section for an explicit construction giving defining
equations.

\begin{thm}\label{examabst} Let $p\equiv 3,5\text{ mod }7$ be a prime number. Take $a\in\Q$ with $a\neq -10,\pm2,-1,0$. Consider the family $C_{p,a}$ of smooth plane curves over $\Q$ given by
$$C_{p,a}: \,X^6+\frac{1}{p^2}Y^6+\frac{1}{p^4}Z^6+\frac{a}{p^3}(p^2X^3Y^3+pX^3Z^3+Y^3Z^3)=0.$$
Then, there exists a twist $C'\in
\operatorname{Twist}_{\Q}(C_{p,a})$ which does not admit a
non-singular plane model over $\Q$.
\end{thm}

\begin{lem} Given $a\neq -10,\pm2,-1,0$ and $\alpha_0\in\overline{\mathbb{Q}}$, each of the curves in the
family
$$C_{\alpha_0,a}:\,X^6+\frac{1}{\alpha_0^2}Y^6+\frac{1}{\alpha_0^4}Z^6+\frac{a}{\alpha_0^3}(\alpha_0^2X^3Y^3+\alpha_0X^3Z^3+Y^3Z^3)=0,$$
has the automorphism
$[Y,Z,\alpha_0X]\in \operatorname{Aut}(\overline{C}_{\alpha_0,a})$.
\end{lem}

\begin{rem}
	Indeed, it is not difficult to prove (see \cite{BaBa2,BaBa3}) that the automorphism group of the curves in the previous family is isomorphic to $\operatorname{GAP}(54,5)$ and generated by the elements $[Y:Z:\alpha_0X]$ together with $[Y:\sqrt[3]{\alpha_{0}^{2}}X:\sqrt[3]{\alpha_{0}}Z]$ and $[X:Y:\zeta_3Z]$ where $\zeta_3$ is a
	primitive $3$-rd root of unity.
\end{rem}

\begin{proof} (of Theorem \ref{examabst}) Consider the Galois extension $M/\Q$ with
$M=\mathbb{Q}(cos(2\pi/7),\zeta_3,\sqrt[3]{p})$ where all the
elements of $\text{Aut}(\overline{C}_{p,a})$ are defined. Let $\sigma$
be a generator of the cyclic Galois group
$\text{Gal}(\mathbb{Q}(cos(2\pi/7))/\mathbb{Q})$. We define a
1-cocycle in $\text{Gal}(M/\mathbb{Q})\cong
\text{Gal}(\mathbb{Q}(cos(2\pi/7))/\mathbb{Q})\times
\text{Gal}(\mathbb{Q}(\zeta_3,\sqrt[3]{p})/\mathbb{Q})$ to
$\text{Aut}(\overline{C}_{p,a})$ by mapping $(\sigma,id)\mapsto [Y,Z,p X]$
and $(id,\tau)\mapsto id$. This defines an element
of $\text{H}^1(\text{Gal}(M/\mathbb{Q}),\text{Aut}(\overline{C}_{p,a}))$.

Consider its image by $\Sigma$ inside
$\text{H}^1(\text{Gal}(M/\mathbb{Q}),\text{PGL}_3(M))$. We need to check that its image
is not the trivial element, and then the result is an
immediate consequence of Theorem \ref{lem1}.

By Theorem \ref{parametrization},
$\text{H}^1(\text{Gal}(M/\mathbb{Q}),\text{PGL}_3(M))$ is the set of
central simple algebras over $\mathbb{Q}$ of dimension 9 which
splits in a degree 3 field inside $M$. If we consider the image in
$\text{H}^1(\text{Gal}(\mathbb{Q}(cos(2\pi/7))/\mathbb{Q}),
\text{PGL}_3(\mathbb{Q}(cos (2\pi/7))))$ then it is non-trivial if
and only if $p$ is not a norm of the field extension
$\mathbb{Q}(cos(2\pi/7))/\mathbb{Q}$.

By \cite[Theorem 2.13]{Wa}, the ideal $(p)$ is prime in
$\mathbb{Q}(cos(2\pi/7))/\mathbb{Q}$, therefore $p$ is not a norm of
an element of $\Q(\cos(2\pi/7))$. Now
$\text{H}^1(\text{Gal}(M/\mathbb{Q}),\text{PGL}_3(M))$ is the union of
the above central simple algebras over $\mathbb{Q}$ running
through the subfields $F\subset M$ of degree 3 over
$\mathbb{Q}$, see \cite{Ja}. Thus the element is not
trivial, which was to be shown.
\end{proof}
\end{section}

\section{An explicit non-trivial Brauer-Severy variety}

In this Section we give another example of a plane curve defined over $k$ having a twist without such a plane model defined over $k$. The interesting point here is that we show explicit equations of the twist as well as equations of the Brauer-Severy variety containing the twist as in Theorem \ref{proprx}.

As fas as we know, this is the first time that this kind of equations are exhibited. Unfortunately, we were not able to find any example defined over the rational numbers $\mathbb{Q}$ and the example is over $k=\mathbb{Q}(\zeta_3)$.

% In this section, we explicitly construct a twist of a smooth plane curve $C_a$ over $\Q(\zeta_3)$ which does not admit a plane model over
% $\Q(\zeta_3)$.
%In particular, we construct a non-trivial Brauer-Severi surface over
%$\Q(\zeta_3)$ where such twist leaves,
% giving its equations inside $\mathbb{P}^9_{\Q(\zeta_3)}$.
%
% This section may not be surprising for the specialists, but we did not find in the
% literature any explicit construction of a non-trivial Brauer-Severi
% surface giving explicitly its equations.

Let us consider the curve
$C_a:\,X^6+Y^6+Z^6+a(X^3Y^3+Y^3Z^3+Z^3X^3)=0$ defined over a number
field $k\supseteq \mathbb{Q}(\zeta_3)$ where $\zeta_3$ is a
primitive third root of unity and $a\in k$. For $a\neq
-10,-2,-1,0,2$, it is a non-hyperelliptic, non-singular plane curve
of genus $g=10$ and its automorphism group is the group of order 54
determined in the previous section.

The algorithm in \cite{Lo}, allows us to
compute all the twists of $C_a$, previous computation of its canonical model in
$\mathbb{P}^9$. We follow such algorithm, since this time we will see that $\Sigma$ is not trivial, so we cannot use the improvements in Remark \ref{rem4.3}.

\subsection{A canonical model of $C_a$ in $\mathbb{P}^9$}

Let us denote by $\alpha_i$, $i\in\{1,...,6\}$, the six different root of the polynomial
 $T^6+aT^3+1=0$, and define the points on $C_a$:
$P_i=(0:\alpha_{i}:1)$, $Q_i=(\alpha_i:0:1)$ and
$\infty_i=(\alpha_i:1:0)$ for $i\in\{1,...,6\}$. The divisor of the function $x=X/Z$ is
$\text{div}(x)=\sum P_i-\sum \infty_i$. Let $P=(X_0:Y_0:1)\in C_a$, the function
$x$ is a uniformizer at $P$ if the polynomial
$T^6+a(X_{0}^{3}+1)T^3+X_{0}^{6}+aX_{0}^{3}+1=0$ does not have
double roots. That is, if $X_{0}^{6}+aX_{0}^{3}+1\neq 0$ or
$4(X_{0}^{6}+aX_{0}^{3}+1)\neq a^2(X_{0}^{3}+1)^2$. Let us denote by
$\beta_i$, $i\in\{1,...,6\}$, the six different roots of the polynomial
$T^6+\frac{2a}{a+2}T^3+1=0$ and denote by
$V_{ij}=(\beta_i:\zeta_{3}^{j+1}\sqrt[3]{-\frac{a}{2}(\beta_{i}^{3}+1)}:1)$
where $j\in \{1,2,3\}$. In order to compute $\text{ord}_P(dx)$ we
need to use the expression
$$
dx=-\frac{y^2}{x^2}\frac{2y^3+a(x^3+1)}{2x^3+a(y^3+1)}dy
$$
for the points $Q_i$ and $V_{i,j}$. Notice that
$\text{div}(2y^3+a(x^3+1))=\sum V_{i,j}-3\sum\infty$. For the points at
infinity, we use that the degree of a differential is $2g-2=18$. We
finally get
$$
\text{div}(dx)=2\sum Q_i+\sum V_{i,j}-2\sum \infty.
$$

Hence, a basis of regular differentials is given by
$$
\omega_1=\omega=\frac{xdx}{y(2y^3+a(x^3+1))},\,\omega_2=\frac{x^2}{y}\omega,\,\omega_3=\frac{y^2}{x}\omega,\,\omega_4=\frac{1}{xy}\omega
$$
$$
\omega_5=x\omega,\,\omega_6=\frac{y}{x}\omega,\,\omega_7=\frac{1}{y}\omega,\,\omega_8=y\omega,\,\omega_9=\frac{x}{y}\omega,\,\omega_{10}=\frac{1}{x}\omega.
$$
We list the divisors of these differentials below.
$$
\text{div}(\omega_1)=\sum P_i+\sum Q_i+\sum\infty_i,\,\text{div}(\omega_2)=3 \sum P_i,\,\text{div}(\omega_3)=3\sum Q_i,\,\text{div}(\omega_4)=3\sum\infty_i
$$
$$
\text{div}(\omega_5)=2\sum P_i+\sum Q_i,\,\text{div}(\omega_6)=2\sum Q_i+\sum\infty_i,\,\text{div}(\omega_7)=\sum P_i+2\sum \infty_i,
$$
$$
\text{div}(\omega_8)=\sum P_i+2\sum Q_i,\,\text{div}(\omega_9)=2\sum P_i+\sum\infty_i,\,\text{div}(\omega_{10})=\sum Q_i+2\sum\infty_i.
$$

\begin{lema}\label{generators} The ideal of the canonical model of ${C}_a$ in $\mathbb{P}^9[\omega_1,\omega_2,\omega_3,\omega_4,\omega_5,\omega_6,\omega_7,\omega_8,\omega_9,\omega_{10}]$ is generated by the polynomials
$$
\omega_4\omega_9=\omega_{7}^{2},\,\omega_4\omega_6=\omega_{10}^2,\,\omega_4\omega_1=\omega_7\omega_{10},\,\omega_4\omega_5=\omega_9\omega_{10},\,\omega_4\omega_8=\omega_6\omega_7,\,\omega_4\omega_2=\omega_7\omega_9,\,\omega_4\omega_3=\omega_6\omega_{10},
$$
$$
\omega_3\omega_{10}=\omega_{6}^2,\,\omega_2\omega_7=\omega_{9}^2,\,\omega_6\omega_9=\omega_{1}^2,\,\omega_3\omega_5=\omega_{8}^2,\,\omega_2\omega_3=\omega_5\omega_8,\,\omega_2\omega_8=\omega_{5}^2,
$$
$$
\omega_{2}^2+\omega_{3}^2+\omega_{4}^2+a(\omega_5\omega_8+
\omega_6\omega_{10}+\omega_7\omega_9)=0.
$$
We denote by $\mathcal{C}_a$ this canonical model.
\end{lema}

\begin{proof} If $\omega_4\neq 0$, then the des-homogenization of this ideal with respect to $\omega_4$ gives the affine curve $C_a$ for $Z=1$.
If $\omega_4=0$, then $\omega_7=\omega_{10}=0$, so
$\omega_6=\omega_9=0$ and $\omega_1=0$, so if $\omega_3\neq 0$ we
recover the part at infinity ($Z=0$) of $C_a$. If
$\omega_4=\omega_3=0$, then all the variables are equal to zero
which produces a contradiction.

To check that it is non-singular, we need to see if the rank of the matrix of partial derivatives of the previous generating functions has rank equal to $8=\text{dim}(\mathbb{P}^9)-\text{dim}(C)$ at every point, that is, that the tangent space has codimension $1$. If $\omega_4\neq 0$, then the partial derivatives of the first seven equation plus the last one produce linearly independent vectors in the tangent space. If $\omega_4=0$, we have already seen that $\omega_3\neq 0$ and by equivalent arguments, neither it is $\omega_2$. Then the $6th$, $7th$, $8th$, $9th$ equations plus the last four equations produce $8$ linearly independent vectors.
\end{proof}

\begin{rem}\label{eqP2} The canonical embedding of $C_a$ in $\mathbb{P}^{g-1}=\mathbb{P}^9$ coincides with the composition of the $g^d_2$-linear system of $C_a$ with
the Veronese embedding given by:
$$
\mathbb{P}^2\hookrightarrow\mathbb{P}^9:\,(x:y:z)\rightarrow(xyz:x^3:y^3:z^3:x^2y:y^2z:z^2x:xy^2:x^2z:yz^2).
$$
In particular, we get that the ideal defining the projective space $\mathbb{P}^2$ in $\mathbb{P}^9$ by the Veronese embedding is generated by the polynomials defined in Lemma \ref{generators} after removing the last one.
\end{rem}

\subsection{The automorphism group of $C_a$ in $\mathbb{P}^9$}

Let us consider the automorphisms of the curve $C_a$ given by
$R=[y:x:z]$, $T=[z:x:y]$ and
$U=[x:y:\zeta_3z]$. We easily check that $<R,T,U>\subseteq\text{Aut}(C_a)$ and by Lemma
\ref{lemaut}, we obtain that $\text{Aut}(C_a)=<R,T,U>$.%,

Notice that the pullbacks $R^*(\omega)=-\omega$,
$T^*(\omega)=\omega$ and
$U^{*}(\omega)=\zeta_3^2\omega$. So, in the canonical model, these
automorphisms look like

$$
R\rightarrow-\mathcal{R}=-\left(\begin{array}{c|ccc|ccc|ccc}1&0&0&0&0&0&0&0&0&0\\ \hline
0&0&1&0&0&0&0&0&0&0\\
0&1&0&0&0&0&0&0&0&0\\
0&0&0&1&0&0&0&0&0&0\\ \hline
0&0&0&0&0&0&0&1&0&0\\
0&0&0&0&0&0&0&0&1&0\\
0&0&0&0&0&0&0&0&0&1\\ \hline
0&0&0&0&1&0&0&0&0&0\\
0&0&0&0&0&1&0&0&0&0\\
0&0&0&0&0&0&1&0&0&0
\end{array}\right),
T\rightarrow\mathcal{T}=\left(\begin{array}{c|ccc|ccc|ccc}1&0&0&0&0&0&0&0&0&0\\
\hline
0&0&0&1&0&0&0&0&0&0\\
0&1&0&0&0&0&0&0&0&0\\
0&0&1&0&0&0&0&0&0&0\\ \hline
0&0&0&0 &0&0&1&0&0&0\\
0&0&0&0 &1&0&0&0&0&0\\
0&0&0&0 &0&1&0&0&0&0\\ \hline
0&0&0&0 &0&0&0&0&1&0\\
0&0&0&0 &0&0&0&0&0&1\\
0&0&0&0 &0&0&0&1&0&0
\end{array}\right)
$$
and $U\rightarrow\zeta_3^2
\text{Diag}(1,\zeta_3^{2},\zeta_3^{2},\zeta_3^{2},\zeta_3^2,1,\zeta_3,\zeta_3^2,1,\zeta_3)=\zeta_3^2\mathcal{U}$.
We define the faithful linear representation
$\text{Aut}(C_a)\hookrightarrow\text{GL}_{10}(\overline{k})$
by sending
$R,T,U\rightarrow\mathcal{R},\mathcal{T},\mathcal{U}$. Moreover, it preserves the action of the Galois group $G_k$.

\subsection{A explicit twist over $k=\Q(\zeta_3)$ of $C_a$ without a non-singular plane model over $k$}\label{twist}

Let us consider the subgroup $N$ of $\text{Aut}(C_a)$ generated by
$N:=<TU>\simeq \Z/3\Z$.

Let us consider the curve ${C}_a$ defined over
$k=\mathbb{Q}(\zeta_3)$, and the field extension $L=k(\sqrt[3]{7})$
with Galois group $\text{Gal}(L/k)=<\sigma>\simeq \Z/3\Z$, where
$\sigma(\sqrt[3]{7})=\zeta_3\sqrt[3]{7}$. We define  the cocycle
$\xi\in\text{Z}^1(G_k,\text{Aut}(C_a))\hookrightarrow
\text{Z}^1(G_k,\text{PGL}_{10}(\overline{k}))$ given by
$\xi_{\sigma}=\mathcal{T}\mathcal{U}$.

\begin{lema} The image of the cocycle $\xi$ by the map $\Sigma:\operatorname{H}^1(\operatorname{G}_k,\operatorname{Aut}(C_a))\rightarrow \operatorname{H}^1(\operatorname{G}_k,\operatorname{PGL}_{3}(\overline{k}))$
is  not trivial.
\end{lema}

\begin{proof}  By construction, the image of the cocycle $\xi$ in $\text{H}^1(k,\text{PGL}_3(\overline{k}))$ coincides with the inflation of the
cocycle in $\text{H}^1(\operatorname{Gal}(L/k),\text{PGL}_3(L))$ where $\xi_\sigma=TU$.
Now by Theorem \ref{parametrization} we conclude, since $\zeta_3$ is
not a norm in $L/k$ (no new primitive root of unity appears in $L$
than $k$ and $\zeta_3$ is not a norm of an element of $L$).
\end{proof}

 We can then take
  $$
  \phi=\left(\begin{array}{c|ccc|ccc|ccc}
  1 &       0 & 0 & 0 &           0 & 0 & 0        & 0 & 0& 0\\ \hline
  0&        \sqrt[3]{7}& \sqrt[3]{7^2}& 7  & 0  & 0 & 0        & 0 & 0& 0\\
  0 &      \sqrt[3]{7}&\zeta_3\sqrt[3]{7^2}& 7\zeta_{3}^2  & 0 & 0 & 0       & 0 & 0& 0\\
  0&       \sqrt[3]{7}&\zeta_{3}^2\sqrt[3]{7^2}& 7\zeta_3  & 0 & 0  & 0      & 0 & 0& 0\\ \hline
  0&    0 & 0 & 0        & 1& \sqrt[3]{7}&\zeta_3\sqrt[3]{7^2}& 0 & 0& 0\\
  0&    0 & 0 & 0        & 1& \zeta_3\sqrt[3]{7}&\sqrt[3]{7^2}& 0 & 0& 0\\
  0&    0 & 0 & 0        & \zeta_3& \sqrt[3]{7}&\sqrt[3]{7^2}& 0 & 0& 0\\ \hline
  0&    0 & 0 & 0   & 0 & 0 & 0   & 1& \zeta_3\sqrt[3]{7}&\zeta_3\sqrt[3]{7^2}\\
  0&    0 & 0 & 0   & 0 & 0 & 0   & \zeta_3& \zeta_3\sqrt[3]{7}&\sqrt[3]{7^2}\\
  0&    0 & 0 & 0   & 0 & 0 & 0   & \zeta_3& \sqrt[3]{7}&\zeta_3\sqrt[3]{7^2}\\
  \end{array}\right)
  $$

If we simply substitute this isomorphism $\phi$ in the equations of
$\mathcal{C}_a$, we will get equations for $\mathcal{C}'_a$. However, even
defining a curve over $k$, this equations are defined over
$L=k(\sqrt[3]{7})$. In order to get generators of the ideal defined
over $k$, we use the following lemma.

\begin{lema}\label{eqoverk} Let $f_0,f_1,f_2\in k[x_1,...,x_n]$, and define $g_0=f_0+\sqrt[3]{7}f_1+\sqrt[3]{7^2}f_2$,  $g_1=f_0+\zeta_{3}\sqrt[3]{7}f_1+\zeta_{3}^2\sqrt[3]{7^2}f_2$, $g_2=f_0+\zeta_{3}^2\sqrt[3]{7}f_1+\zeta_3\sqrt[3]{7^2}f_2$. Then the ideals in $L[x_1,...,x_n]$ generated by $<g_0,g_1,g_2>$ and $<f_0,f_1,f_2>$ are equal.
\end{lema}

\begin{proof}
Clearly, we have the inclusion $<g_0,g_1,g_2>\subseteq<f_0,f_1,f_2>$. The reverse inclusion can be checked by writing $3f_0=g_0+g_1+g_2$, $(\zeta_3-1)\sqrt[3]{7}f_1=g_1-\zeta_3g_2+(\zeta_3-1)f_0$ and $\sqrt[3]{7^2}f_2=g_0-f_0-\sqrt[3]{7}f_1$.
\end{proof}

\begin{proposition}\label{ntBSsurface} The equations in $\mathbb{P}^9$ of the non-trivial Brauer-Severi surface $B$ over $k$ constructed as in Theorem \ref{lem1} from
    the cocycle $\xi$ above are
    $$
    \begin{array}{ll} \omega_1\omega_2=\zeta_3\omega_5\omega_9+\zeta_3\omega_6\omega_8+7\zeta_3\omega_7\omega_{10}, &
    \omega_{2}^{2}-7\omega_3\omega_4=\zeta_3\omega_5\omega_{10}+\zeta_3\omega_7\omega_8+\zeta_3\omega_6\omega_9,\\
    \omega_1\omega_3=\omega_5\omega_{10}+\zeta^2\omega_7\omega_8+\zeta_3\omega_6\omega_9, & 7\omega_{3}^{2}-7\zeta_3\omega_2\omega_4=\omega_5\omega_9
    +\zeta_3^2\omega_6\omega_8+7\zeta_3\omega_7\omega_{10},\\
    7\omega_1\omega_4=\zeta_3\omega_5\omega_8+7\omega_6\omega_{10}+7\zeta_3^2\omega_7\omega_9,
    &
    49\omega_{4}^{2}-7\zeta_3^2\omega_2\omega_3=\omega_5\omega_8+7\zeta_3\omega_6\omega_{10}+7\zeta_3^2\omega_7\omega_9,
    \end{array}
    $$

$$
\begin{array}{ll} \omega_{5}^{2}+14\zeta_3\omega_6\omega_7=7\zeta_3\omega_2\omega_{10}+7\omega_4\omega_8+7\zeta_3\omega_3\omega_9, &\omega_{5}^{2}-7\zeta_3\omega_6\omega_7=7\omega_2\omega_{10}+7\omega_4\omega_8+7\zeta_3^{2}\omega_3\omega_9, \\
\omega_{6}^{2}+2\zeta_3\omega_5\omega_7=\zeta_3\omega_2\omega_9+\omega_3\omega_8+7\zeta_3\omega_4\omega_{10}, & \omega_{6}^{2}-\zeta_3\omega_5\omega_7=\omega_2\omega_9+\zeta_3\omega_3\omega_8+7\zeta_3\omega_4\omega_{10},\\
7\omega_{7}^{2}+2\zeta_3\omega_5\omega_6=\zeta_3\omega_2\omega_8+7\zeta_3^2\omega_3\omega_{10}+7\zeta_3^2\omega_4\omega_9, & 7\omega_{7}^{2}-\zeta_3\omega_5\omega_6=\omega_2\omega_8+7\omega_3\omega_{10}+7\zeta_3^2\omega_4\omega_9,\\
\end{array}
$$

$$
\begin{array}{ll}
\omega_{8}^{2}+14\zeta_3\omega_9\omega_{10}=7\zeta_3^2\omega_2\omega_7+\omega_4\omega_5+7\zeta_3^2\omega_3\omega_6, &\omega_{8}^{2}-7\zeta_3^2\omega_9\omega_{10}=7\zeta_3^2\omega_2\omega_7+7\zeta_3^{2}\omega_4\omega_5+7\omega_3\omega_6, \\
\omega_{9}^{2}+14\zeta_3^2\omega_8\omega_{10}=\zeta_3\omega_2\omega_6+\zeta_3^2\omega_3\omega_5+7\zeta_3\omega_4\omega_7,&\omega_{9}^{2}-7\zeta_3^2\omega_8\omega_{10}=\zeta_3^2\omega_2\omega_6+\zeta_3\omega_3\omega_5+7\zeta_3\omega_4\omega_7,\\
7\omega_{10}^{2}+2\zeta_3^2\omega_8\omega_9=\zeta_3^2\omega_2\omega_5+7\omega_3\omega_7+7\omega_4\omega_6,& 7\omega_{10}^{2}-\zeta_3^2\omega_8\omega-9=\zeta_3^2\omega_2\omega_5+7\zeta_3^2\omega_3\omega_7+7\omega_4\omega_6,\\
\end{array}
$$
\end{proposition}

\begin{proof}
We only need to plug the equations of the isomorphism $\phi$ into the equations defining $\mathcal{C}_a$ and apply Lemma \ref{eqoverk}.
\end{proof}

In order to get the equations of the twisted curve, we only need to add the equation that we get by plugging $\phi$ in $\omega_{2}^2+\omega_{3}^2+\omega_{4}^2+a(\omega_5\omega_8+
\omega_6\omega_{10}+\omega_7\omega_9)=0,
$ and apply Lemma \ref{eqoverk} again.

\begin{proposition} The curve $\mathcal{C}'_a$ is a twist over $k$
of the curve $C_a$ for $a\neq -10,-2,-1,0,2$ which does not admits a
non-singular plane model over $k$, i.e. is not a smooth plane curve
over $k$, and the defining equations of $\mathcal{C}'_a$ in
$\mathbb{P}^9$ are the ones given in Proposition \ref{ntBSsurface}
plus the extra equation:
$$
\omega_{2}^{2}+14\omega_3\omega_4+a(\omega_{2}^2-7\omega_3\omega_4)=0
$$
\end{proposition}

\subsection{A counterexample of the Hasse Principle for plane models}

It is well-known that $\operatorname{H}^1(\F_q,\operatorname{PGL}_n(\overline{\F_q}))$ is
trivial, therefore all twists of a smooth plane curve over $\F_q$
are smooth plane curves over $\F_q$. Here we work with the previous
example in order to show examples of the smooth plane reduction over $\F_q$ of the twist. It is interesting to mention at this point that we can see this twist like a Hasse Principle counterexample for having a plane model: the twist is defined over $k$, it does not a have a plane model defined over $k$, but it does over $k\otimes\mathbb{R}=\mathbb{C}$ and for all the (good) reductions modulo a prime number $p$.

We consider the reductions $\tilde{C}_a$ and $\tilde{\mathcal{C}}_a$
at a prime $\mathfrak{p}$ of good reduction of the curve $C_a/k$ and the
twist $\mathcal{C}'_a/k$ computed in subsection \ref{twist}. Since
$k=\Q(\zeta_3)$, the resulting reductions of the curves are defined over a
finite field $\mathbb{F}_q$ with $q\equiv 1\text{ mod }3$, and
$q=p^f$ for some $f\in\mathbb{N}$ and $\mathfrak{p}\mid p$. We also
assume that $p>21=(6-1)(6-2)+1$ in order to ensure that
$\text{Aut}(\tilde{C}_a)\simeq<54,5>$, see \cite[\S6]{BaBa1}.

The natural map $\operatorname{G}_k\rightarrow \operatorname{G}_{\mathbb{F}_q}$ induces a map
$\text{H}^1(k,\text{Aut}(C_a))\rightarrow
\text{H}^1(\mathbb{F}_q,\text{Aut}(\tilde{C}_a))$. Since
$\text{Z}^1(\mathbb{F}_q,\text{Aut}(\tilde{C}_a))\hookrightarrow\text{Z}^1(\operatorname{G}_{\mathbb{F}_q},\text{PGL}_3(\overline{\mathbb{F}_q}))$
and
$\text{H}^1(\operatorname{G}_{\mathbb{F}_q},\text{PGL}_3(\overline{\mathbb{F}_q}))=1$,
the reduction of the twist is a smooth plane curve over $\F_q$.

Clearly, if $7\in\mathbb{F}_{q}^{3}$, then the twist $\mathcal{C}'_a$ becomes trivial.
Otherwise, we get that the reduction of the cocycle $\xi$ is given by its image at $\pi$,
the Frobenius endomorphism \cite{MT}, and $\xi_{\pi}$ can take the values

$$
\left(\begin{array}{ccc}0&0&\zeta_3^{e}\\1&0&0\\0&1&0\end{array}\right)^e,
$$

where $e=0,1,2$ according to the splitting behaviour of the prime $\frak{p}$ in $L=k(\sqrt[3]{7})$. In the first case, we
get the trivial twist. In the later and the former, let assume $e=1$  (the other can be treat symmetrically)
 and $q\not\equiv 1\text{ mod }9$, we can then take a generator $\eta$ of $\mathbb{F}_{q^3}/\mathbb{F}_q$,  such that $\eta^3=\zeta_3$.
Then, the cocycle is given ($\xi_{\sigma}=\phi\,^{\sigma}\phi^{-1}$)
by the isomorphism
$$
\phi=\left(\begin{array}{ccc}1&\eta&\eta^2 \\
\eta^2&\zeta_3^2&\eta\\
\eta&\zeta_3^2\eta^2&\zeta_3^2\end{array}\right):\,\tilde{C}'_a\rightarrow
\tilde{C}_a,
$$
and the twist $\tilde{\mathcal{C}}'_a$ has a non-singular plane
model
$$
\tilde{C}'_a:\,2(x^5z+z^5y+\zeta_3
y^5x)-5(1+\zeta_3)(y^4z^2+x^2z^4)+9x^4y^2+20\zeta_3(x^3yz^2+x^2y^3z)-20(\zeta_3+1)xy^2z^3+
$$
$$
+a(-(x^5z+z^5y+\zeta_3
y^5x)+2x^4y^2-2(\zeta_3+1)(x^2z^4+y^4z^2)-\zeta_3(x^3yz^2-x^2y^3z)+(\zeta_3+1)xy^2z^3).
$$

If $q\equiv 1\text{ mod }9$, the same $\phi$ works, but this time the cocycle becomes trivial since $\eta\in\mathbb{F}_q$.

\end{section}

\begin{section}{Twists of smooth plane curves with diagonal cyclic automorphism group}
We observed in Remark \ref{rem4.3}, that the algorithm for
computing $\operatorname{Twist}_k(C)$ described in \cite{Lo} can be substantially
improved if the smooth curve $C$ over $k$ admits a non-singular plane model
and such that the morphism $\Sigma$ in theorem \ref{lem1} is
trivial.

%In this section, we prove a theoretical result to obtain directly
%all the twists for smooth plane curves over $k$ having an extra
%property: $C$ $k$-isomorphic to the plane $k$-model
%$F_{\overline{C}}(X,Y,Z)=0$ has diagonal cyclic automorphism group,
%i.e., we have that $\text{Aut}(F_{\overline{C}})=<\alpha>$ with
%$\alpha$ a diagonal matrix. In such case, we prove that all the
%elements of $\operatorname{Twist}_k(F_{\overline{C}}=0)$ are given by non-singular
%plane models of the form $F_{D\overline{C}}=0$ with $D$ a diagonal
%matrix. We apply this result to some particular families of smooth
%plane curves over $k$ with large automorphism group which are not
%Fermat or Klein type.
In this section, we prove a theoretical result, by which we obtain directly
all the twists for smooth plane curves $C$ over $k$ having the extra
property: $C$ is isomorphic over $k$ to a plane $k$-model
$F_{\overline{C}}(X,Y,Z)=0$, such that $Aut(F_{\overline{C}})$ is cyclic and generated by an automorphism $\alpha\in PGL_{3}(\overline{k})$ of a diagonal shape. In this case, we show that any twist in $\operatorname{Twist}_k(F_{\overline{C}}=0)$ is represented by a non-singular
plane model of the form $F_{D\overline{C}}=0$ for some diagonal $D\in PGL_{3}(\overline{k})$.
We apply this result to some particular families of smooth
plane curves over $k$ with large automorphism group, different from the Fermat curve and the Klein curve.

\begin{defn} Consider a smooth plane curve $C$ over $k$ given by $F_{\overline{C}}(X,Y,Z)=0$.
We say that $[C']\in \operatorname{Twist}_k(C)$ is a diagonal twist of $C$, if there
exists an $M\in \text{PGL}_3(k)$ and a diagonal $D\in\text{PGL}_3(\overline{k})$, such that $C'$ is $k$-isomorphic to
$F_{(MD)^{-1}\overline{C}}(X,Y,Z)=0$.
\end{defn}

The condition that $\alpha$ is a diagonal matrix is necessary, and
we will provide examples when $\alpha$ is not diagonal, such that
not all the twists are diagonal ones.

\subsection{Diagonal cyclic automorphism group: all twists are diagonal}

\mbox{}
\newline

Motivated by the results in Section \ref{Sec4} and following the
philosophy of the third author's thesis in \cite{Loth}, we prove the
next result.

\begin{thm}\label{lem4} Let $C:\,F_{\overline{C}}(X,Y,Z)=0$ be a smooth plane curve over $k$.
Assume that $\operatorname{Aut}(F_{\overline{C}})\subseteq \operatorname{PGL}_3(\overline{k})$ is a
    non-trivial cyclic group of order $n$ (relatively prime with the characteristic of $k$), generated by an automorphism $\alpha=\operatorname{diag}(1,\zeta_n^a,\zeta_n^b)$ for some $a,b\in\mathbb{N}$.

    Then all the twists in 
    $\operatorname{Twist}_k(C)$ are given by plane equations of the form $F_{D^{-1}\overline{C}}(X,Y,Z)=0$ with
    $F_{D^{-1}\overline{C}}(X,Y,Z)\in k[X,Y,Z]$ and $D$ is a diagonal matrix.
    In particular, the map $\Sigma$ is trivial.
\end{thm}

\begin{proof}
    We just need to notice that the map $\Sigma$ in Theorem \ref{lem1} factors as follows:
    $$
    \Sigma: \,\text{H}^1(k,\text{Aut}(F_{P^{-1}\overline{C}}))\rightarrow (\text{H}^1(k,\text{GL}_1(\overline{k})))^3\rightarrow\text{H}^1(k,\text{GL}_3(\overline{k}))\rightarrow\text{H}^1(k,\text{PGL}_3(\overline{k})).
    $$
    Hence, $\Sigma$ is trivial and all the cocycles are given by diagonal matrices.
\end{proof}

\begin{rem} More generally,
suppose that $C$ is a smooth plane curve over $k$, identified with
$F_{\overline{C}}(X,Y,Z)=0$, and having a twist $[C']\in
\operatorname{Twist}_k(C)$ with a non-singular plane model
$F_{Q\overline{C}}(X,Y,Z)=0$ over $k$ for some $Q\in PGL_{3}(\overline{k})$, such that $\operatorname{Aut}(F_{Q\overline{C}})=\langle
\operatorname{diag}(1,\zeta_n^a,\zeta_n^b)\rangle$. Then, any other twist $[C'']\in
\operatorname{Twist}_k(C)$ is represented by a model $F_{(QD)^{-1}\overline{C}}=0$ over $k$ through some diagonal $D\in PGL_3(\overline{k})$.
\end{rem}

Now, we apply Theorem \ref{lem4} to some particular smooth plane
curves with cyclic automorphism group in order to obtain all of their
twists: let $\mathcal{M}_g$ be the moduli space of smooth curves of genus $g$ over $\overline{k}$. For a finite group $G$, we define the stratum
$\widetilde{\mathcal{M}_g^{Pl}(G)}$ of all smooth $\overline{k}$-plane curves $C$ in $\mathcal{M}_g$, whose full automorphism group is
isomorphic to $G$. In particular, $\widetilde{\mathcal{M}_g^{Pl}(G)}$ is the
disjoint union of the different components
$\rho(\widetilde{\mathcal{M}_g^{Pl}(G)})$, where $\rho:\,G\hookrightarrow PGL_{3}(\overline{k})$ assumes all the possible values of injective representations of $G$ inside
$\text{PGL}_3(\overline{k})$, modulo conjugation. For more details, one can read \cite[Lemma 2.1]{BaBa1}. Moreover, if $k$ is a field of characteristic $p=0$, then we have, by \cite[Theorem 1]{BaBa2}:
\begin{eqnarray*}
\widetilde{\mathcal{M}_g^{Pl}(\Z/d(d-1)\Z)}&=&\{X^d+Y^{d}+XZ^{d-1}=0\},\\
\widetilde{\mathcal{M}_g^{Pl}(\Z/(d-1)^2\Z)}&=&\{X^d+Y^{d-1}Z+XZ^{d-1}=0\}.
\end{eqnarray*} 
Both curves are defined over $k$ with cyclic diagonal automorphism groups of orders $d(d-1)$ and $(d-1)^2$, which are generated by 
$
\operatorname{diag}(1,\zeta_{d(d-1)}^{d-1},\zeta_{d(d-1)}^d),\text{ and }\operatorname{diag}(1,\zeta_{(d-1)^2},\zeta_{(d-1)^2}^{(d-1)(d-2)}),
$
respectively. The same result remains true for positive characteristic $p>(d-1)(d-2)+1$, see, for example, \cite[\S6]{BaBa1}. Furthermore, applying the Theorem, we obtain:

\begin{cor}\label{thm1} Let $k$ be a field of characteristic zero or $p>(d-1)(d-2)+1$. For $d\geq5$, the set $\operatorname{Twist}_{k}(C)$ of  $C:\,X^d+Y^d+XZ^{d-1}=0$ is in bijection with
$\mathcal{A}_1:=(k^*\setminus k^{*^d})\times (k^*\setminus
k^{*^{d-1}})/\sim$, where $(a,b)\sim (a',b')$ iff $a'=q^da,\,b'=qq'^{d-1}b$ for some $q,q'\in k$. A representative twist for the class of $[(a,b)]\in\mathcal{A}_1$ is $aX^d+Y^d+bXZ^{d-1}=0$.

Similarly, for $C:\,X^d+Y^{d-1}Z+XZ^{d-1}=0$, the set $\operatorname{Twist}_{k}(C)$ is in bijection with $\mathcal{A}_2:=(k^*\setminus k^{*^{d-1}})\times(k^*\setminus
k^{*^{d-1}})/\sim$, such that $(a,b)\sim (a',b')$ iff
$a'=q^{d-1}q'a,\,b'=q^{d-1}b$ for some $q,q'\in k$. A representative for each twist class is $X^d+MY^{d-1}Z+NXZ^{d-1}=0$.
\end{cor}

\begin{rem} Fix an injective representation $\rho:\,\Z/n\Z\hookrightarrow PGL_{3}(\overline{k})$, such that $\rho(\Z/n\Z)$ is diagonal. The associated normal forms defining the stratum $\rho(\widetilde{\mathcal{M}_g^{Pl}(\Z/n\Z)})$, are already given in \cite{BaBa3}, for smooth plane curves of genus $6$, and in \cite{BaBa2} for higher genera. Moreover, if we follow the ideas of Lercier,
Ritzenthaler, Rovetta and Sisjling in \cite{LeRiRo}, to study the existence of complete, and representative families over $k$, which parameterizes such strata, then we could apply Theorem \ref{lem4} to have a very nice description of $Twist_k(C)$ for any $C$ in $\rho(\widetilde{\mathcal{M}_g^{Pl}(\Z/n\Z)})$.

In this sense, we provide parameters, in the upcoming work \cite{BaBa4}, for the moduli space of smooth plane curves of genus $6$. In particular, for each substratum $\rho(\widetilde{\mathcal{M}_6^{Pl}(G)})$ when it is non-empty.
\end{rem}

\subsection{$\text{Aut}(C)$ cyclic does not imply diagonal twists}

Let $C$ be a smooth plane curve over
$k$, a field of characteristic zero, and identify $C$ with its 
model $F_{\overline{C}}(X,Y,Z)=0$ over $k$. Suppose also that
$\operatorname{Aut}(F_{\overline{C}})\subseteq \operatorname{PGL}_3(\overline{k})$
is a cyclic group of order $n$, generated by a matrix $\alpha$, such that the conjugacy class of 
$\alpha$ in $\operatorname{PGL}_3(k)$ contains no elements of a diagonal shape. Then the twists of $C$ mapped to zero by
$\Sigma$ (i.e., those ones that admits a smooth plane curve over $k$),
are not necessarily represented by diagonal twists.

\begin{example}\label{nondiagonaltwist}
Consider the following smooth plane curve $C$ over $\Q$:
$$F_{\overline{C}}(X,Y,Z)=X^4Y+Y^4Z+XZ^4+(X^3Y^2+Y^3Z^2+X^2Z^3)=0$$

\noindent \emph{Claim $1$}:\,$\operatorname{Aut}(F_{\overline{C}})=\Z/3\Z$, and it is generated by $[Y:Z:X]$ in $\operatorname{PGL}_3(\overline{\Q})$.

\begin{proof}
Since, $\alpha:=[Y:Z:X]\in \text{Aut}(F_{\overline{C}})$, then $\text{Aut}(F_{\overline{C}})$ is conjugate to one of the groups $Gs$ in \cite[Table 2]{BaBa2}, with $3$ divides the
order. 

First, assume that $\beta\in
\text{Aut}(F_{\overline{C}})$ is of order $2$ with $\beta\alpha\beta=\alpha^{-1}$. Then, make a change of the variables of the shape $[X+Y+Z:X+\zeta_3Y+\zeta_3^2Z:X+\zeta_3^2Y+\zeta_3Z]$, to obtain a $\overline{\mathbb{Q}}$-equivalent model of the form 
$$F_{P^{-1}\overline{C}}=4X^5+20X^3YZ+\left((-5-9 i \sqrt{3})
   Y^3+(-5+9 i \sqrt{3})Z^3\right) X^2-6
   XY^2Z^2-4YZ(Y^3+Z^3)=0,$$ 
such that $P^{-1}\alpha P=\operatorname{diag}(1,\zeta_3,\zeta_3^2)\in
\text{Aut}(F_{P^{-1}\overline{C}})$. So $P^{-1}\beta P\in
\text{Aut}(F_{P^{-1}\overline{C}})$ should be $[X:aZ:a^{-1}Y]$ for some $a\in
\overline{\mathbb{Q}}$, a contradiction (no such isomorphism retains the
defining equation $F_{P^{-1}\overline{C}}=0$). Hence, $S_3$ does not happen 
as a bigger subgroup of automorphisms. Then so do $\operatorname{GAP}(30,1)$ and $\operatorname{GAP}(150,5)$, since both groups contain an $S_3$ and
a single conjugacy class of elements of order $3$. 

\par Second, any automorphism of order $3$ of the $\operatorname{GAP}(39,1)$ in \cite[Table 2]{BaBa2} is conjugate % in Gap(39,1)
to either $\alpha$ or $\alpha^{-1}$. Therefore, if $\text{Aut}(F_{\overline{C}})$ is conjugate, through some
$P\in \text{PGL}_3(\overline{\Q})$, to $\operatorname{GAP}(39,1)$, then we may impose that $P^{-1}\alpha P=\alpha$. Thus 
$P$ reduces to       
$$\left(
    \begin{array}{ccc}
      1 & 0 & 0 \\
      0 & \zeta_3 & 0 \\
      0 & 0 & \zeta_3^2 \\
    \end{array}
  \right)^r
\left(
          \begin{array}{ccc}
            \alpha_1 & \alpha_2 & \alpha_3 \\
            \alpha_3 & \alpha_1 & \alpha_2 \\
            \alpha_2 & \alpha_3 & \alpha_1 \\
          \end{array}
        \right)\in \text{PGL}_3(\overline{\mathbb{Q}}),
      $$
and, $F_{P^{-1}\overline{C}}=\,X^4Y+Y^4Z+Z^4X=0$. We easily deduce that no such $P$ transforms $F_{\overline{C}}=0$ to the mentioned form: indeed, for $r=0$, the coefficients of $Y^5,
Y^4X,Y^3X^2,Y^3Z^2,$ and $Y^3XZ$ should be all zeros, and $P$ is not invertible anymore. For $r=1$ and $2$, we also need to delete the monomials $X^5,Y^5$ and $Z^5$, in particular $P$ is diagonal and $F_{P^{-1}\overline{C}}$ is not $X^4Y+Y^4Z+Z^4X$. Consequently, $\text{Aut}(F_{\overline{C}})$ can not be conjugate to $\operatorname{GAP}(39,1)$. 

As a consequence of the above argument, 
$\text{Aut}(F_{\overline{C}})$ is cyclic of order $3$, which was to be shown.

\end{proof}

\emph{Claim $2$:}\,There exists a twist of $C$ over $\Q$, which is not diagonal.
\begin{proof}
The defining equation $F_{\overline{C}}=0$ has degree $5$, thus
any twist of $\overline{C}$ admits also a non-singular plane model
over $\mathbb{Q}$ defined by $F_{M^{-1}\overline{C}}(X,Y,Z)=0$ for
some $M\in\text{PGL}_3(\overline{\mathbb{Q}})$.

We construct the twist following the algorithm in
\cite{Lo} and Theorem \ref{lem1} because $\Sigma$ is trivial: The curve $F_{\overline{C}}=0$ has exactly two non-trivial
twists for each cyclic cubic field extension $L/\mathbb{Q}$. Since
the set of such extensions is not empty, the curve $\overline{C}$
has a non-trivial twist. However, it is easy to check, that a twist of $F_{\overline{C}}$
through a diagonal isomorphism $D\in
\text{PGL}_3(\overline{\mathbb{Q}})$ is always the trivial one. Therefore, any non-trivial twist of $C$ must be a non-diagonal twist.
\end{proof}

\end{example}

\begin{rem} Example \ref{nondiagonaltwist} extends to any field $k$ of characteristic $p>13$, since Claim $1$ holds by our discussion in \cite[\S6]{BaBa1}. And, we ask for $\zeta_3\notin k$ in order to construct a non-trivial twist as in Claim $2$.
\end{rem}

\begin{rem}
    Degree $5$ is the smallest degree for which such an example exists, see the third author thesis \cite{Loth} to discard degree $4$ exceptions.
\end{rem}

\end{section}


\begin{thebibliography}{10}
\bibitem{Es} E. Badr, \emph{On smooth plane curves}, PhD in
progress at UAB.
\bibitem{BaBa1} E. Badr and F. Bars, \emph{On the locus of smooth plane curves with a fixed automorphism
group}, Mediterr. J. Math. \textbf{13} (2016), 3605-3627. doi:\,10.1007/s00009-016-0705-9.

\bibitem{BaBa3}E. Badr and F. Bars, \emph{Automorphism groups of non-singular plane curves of degree 5 },
Commun. Algebra \textbf{44} (2016), 327-4340.
doi:\,10.1080/00927872.2015.1087547. 

\bibitem{BaBa2}E. Badr and F. Bars, \emph{Non-singular plane curves with an element of ``large" order in its automorphism group }, Int. J. Algebra Comput. 26 (2016), 399-434. doi:\,
10.1142/S0218196716500168.

\bibitem{BaBa4} E. Badr and E. Lorenzo, \emph{Parametrizing the moduli space of smooth plane curves of genus $6$}. In preperation.

\bibitem{Chang}H. C. Chang, \emph{On plane algebraic curves}, Chinese J. Math. \textbf{6} (1978), no. 2, 185-189.

\bibitem{Cha} F. Ch$\hat{a}$telet, \emph{Variations sur un th\`eme de H. Poincar\'e}, Ann. Sci. Ec. Norm. Sup.
61 (1944), 249-300.

\bibitem{GAP}  The GAP Group, GAP -Groups, Algorithms,  and Programming, Version 4.5.7, 2012. (http://www.gap-system.org)

\bibitem{Graaf} W. A. de Graaf, M. Harrison, J. Pilnikova, and J. Schicho, \emph{A Lie algebra method for
rational parametrization of Severi-Brauer surfaces}. arXiv:math/0501157v2 [math.AG]


\bibitem{Ha}T. Harui, \emph{Automorphism groups of plane
curves}, arXiv: 1306.5842v2[math.AG] 7 Jun 2014

\bibitem{Book} J. W. P. Hirschfeld, G. Korchm\'aros, and F. Torres,
\emph{Algebraic Curves over Finite Fields}, Princeton Series in
Applied Mathematics, 2008.

\bibitem{Hug} B. Huggins, \emph{Fields of moduli and fields of definition of
curves}. PhD thesis, Berkeley (2005), see
http://arxiv.org/abs/math/0610247v1.

\bibitem{Ja} J. Jahnel, \emph{The Brauer-Severi variety associated with a
central simple algebra: a survey}. See the book in
https://www.math.uni-bielefeld.de/lag/man/052.pdf

\bibitem{LeRi} R. Lercier, C. Ritzenthaler, \emph{Hyperelliptic curves and
their invariants: geometric, arithmetic and algorithmic aspects}.
J. Algebra 372 (2012),
595-636.

\bibitem{LeRiRo} R. Lercier, C. Ritzenthaler, F. Rovetta, J. Sijsling, \emph{Parametrizing the moduli space of curves and applications to smooth plane quartics over finite fields}.
(LMS Journal of Computation and Mathematics, Volume 17, Special
Issue A (ANTS XI), LMS, London, pp. 128--147, 2014

\bibitem{Loth} E. Lorenzo, \emph{Arithmetic properties of non-hyperelliptic genus 3
curves}. PhD dissertation, Universitat Polit\`ecnica de Catalunya
(2015), Barcelona.

\bibitem{Lo} E. Lorenzo, \emph{Twist of non-hyperelliptic curves},
arXiv:1503.03281, to appear in Revista Matem\'atica Iberoamericana.

\bibitem{MT} S. Meagher, J. Top, \emph{Twists of genus three curves over finite fields}. Finite fields and their applications. 16, 5, p. 347-368, 2010.

\bibitem{Pa} R. Pannekoek, \emph{On the parametrization over $\mathbb{Q}$
of cubic surfaces}. Master dissertation, May 2009, University of
Groningen.

\bibitem{RoXa} J. Ro\'e, X. Xarles, \emph{Galois descent for the gonality of curves},
Arxiv:1405.5991v3 (2015).

\bibitem{Se} J.P. Serre, \emph{Cohomologie Galoisianne}, LNM 5, Springer .

\bibitem{SeL} J.P. Serre, \emph{Local Fields}, GTM , Springer.

\bibitem{SuVo} J. Sijsling, J.Voight, \emph{On explicit descent of marked
curves and map}, arXiv:1504.02814v2 (2015).

\bibitem{Ten} E. Tengan, \emph{Central Simple Algebras and the Brauer
group} , XVII Latin American Algebra Colloquium, (2009). See the
book in http://www.icmc.usp.br/~etengan/algebra/arquivos/cft.pdf

\bibitem{Wa} L. C. Washington, \emph{Introduction to cyclotomic fields}, GTM
83, Second edition, Springer, (1997).

\bibitem{Wed} J. H. M. Wedderbum, \emph{On division algebras}. Trans.Amer.Math.Soc.22(1921).

\bibitem{We} A. Weil, \emph{The field of definition of a variety}. American J. of Math. vol. 78, nï¿½3 (1956),
509--524.

\end{thebibliography}
\end{document}